\documentclass{amsart}
\usepackage{amsmath,amssymb, amscd, stmaryrd}
\usepackage{fullpage}
\usepackage{enumerate}
\usepackage{cite}
\usepackage{url}
\usepackage{framed}
\usepackage[dvips,dvipdf]{graphicx}
\usepackage[usenames,dvipsnames]{color}
\usepackage{color}
\usepackage{tabularx}
\usepackage{array}
\usepackage{multirow}
\usepackage{changes}
\usepackage{dcolumn}
\usepackage{esint}

\newcolumntype{d}[1]{D..{#1}}

\newtheorem{theorem}{Theorem}[section]
\newtheorem{proposition}[theorem]{Proposition}

\theoremstyle{definition}

\theoremstyle{definition}

\theoremstyle{remark}
\newtheorem{remark}[theorem]{Remark}
\newtheorem{example}[theorem]{Example}

\newcommand{\CC}{\mathbb C}

\newcommand{\RR}{\mathbb R}

\newcommand{\ii}{\mathfrak i}

\newcommand{\floor}[1]{\left\lfloor#1\right\rfloor} 
\renewcommand\hat\widehat
\renewcommand\tilde\widetilde

\renewcommand{\Re}{\text{Re}\,}
\renewcommand{\Im}{\text{Im}\,}

\author[J. Ovall]{Jeffrey S. Ovall}
\address{Jeffrey S. Ovall,
  Fariborz Maseeh Department of Mathematics and Statistics,
  Portland State University,
  Portland, OR 97201}
\email{jovall@pdx.edu}

\author[S. Reynolds]{Samuel E. Reynolds}
\address{Samuel Reynolds,
	Fariborz Maseeh Department of Mathematics and Statistics,
	Portland State University,
	Portland, OR 97201}
\email{ser6@pdx.edu}

\thanks{This work was partially supported by the
	National Science Foundation through
	NSF grant DMS-2012285 and
	NSF RTG grant DMS-2136228.}

\begin{document}
\title[Cells with Holes]{
		Evaluation of Inner Products of Implicitly-defined Finite Element
		Functions on Multiply Connected Planar Mesh Cells
	}
\date{March 13, 2023}

\begin{abstract}
	Recent advancements in finite element methods allows for the
	implementation of mesh cells with curved edges.
	In the present work, we develop the tools necessary to
	employ multiply connected mesh cells, i.e. cells with holes,
	in planar domains.
	Our focus is efficient evaluation the $H^1$ semi-inner product
	and $L^2$ inner product of implicitly-defined finite element functions
	of the type arising in
        boundary element based finite element methods (BEM-FEM) and
        virtual element methods (VEM).
	These functions may be defined by specifying a polynomial Laplacian
	and a continuous Dirichlet trace.
	We demonstrate that these volumetric integrals can be reduced to integrals
	along the boundaries of mesh cells, thereby avoiding the need to perform
	any computations in cell interiors.
	The dominating cost of this reduction is solving
	a relatively small Nystr\"om system to obtain a Dirichlet-to-Neumann map,
	as well as the solution of two more Nystr\"om systems to obtain an
	``anti-Laplacian'' of a harmonic function,
	which is used for computing the $L^2$ inner product.
	We demonstrate that high-order accuracy can be achieved with several
	numerical examples.
\end{abstract}

\maketitle

\section{Introduction}
	\label{introduction-section}

	Let $K\subset\RR^2$ be an open, bounded, and connected planar region
	with a piecewise $C^2$ smooth boundary $\partial K$.
	Assume the boundary $\partial K$ is partitioned into a finite number of
	\textit{edges}, with each edge being $C^2$ smooth and connected.
	Edges are permitted to meet at interior angles strictly between
	$0$ and $2\pi$, so that $\partial K$ has no cusps or slits.
	Consider the problem of computing
	the $H^1$ semi-inner product
	and $L^2$ inner product
	\begin{align}
		\label{goal-integral-1}
		\int_K \nabla v \cdot \nabla w~dx~,
		\\
		\label{goal-integral-2}
		\int_K v \, w~dx~,
	\end{align}
	where $v,w$ are implicitly defined elements of a
	\textit{local Poisson space} $V_p(K)$, which we define as follows.
	Fix a natural number $p$, and let $V_p(K)$ consist of the functions
	$v\in H^1(K)\cap C^2(K)$ such that:
	\begin{enumerate}
		\item for $p=1$, $v$ is harmonic in $K$;
		\item for $p \geq 2$, the Laplacian $\Delta v$ is a
			polynomial of degree at most $p-2$ in $K$;
		\item the trace $v|_{\partial K}$ is continuous;
		\item the trace $v|_e$ along any edge $e\subset\partial K$ is the
			trace of a polynomial of degree at most $p$
			(defined over all of $\RR^2$).
	\end{enumerate}
	Note, for instance, that $V_p(K)$ contains all of the polynomials
	of degree at most $p$.
	Such subspaces of $H^1(K)$ arise naturally in the context of
	finite element methods posed over
	\emph{curvilinear meshes}, whose mesh cells have curved edges.
	Our present interest is extending the application of
	theses spaces to curvilinear meshes with \emph{punctured}
	(i.e. multiply connnected)
	mesh cells; see Figure \ref{mesh-examples-figure}.

	\begin{figure}
		\centering
		\includegraphics[width=0.25\linewidth]{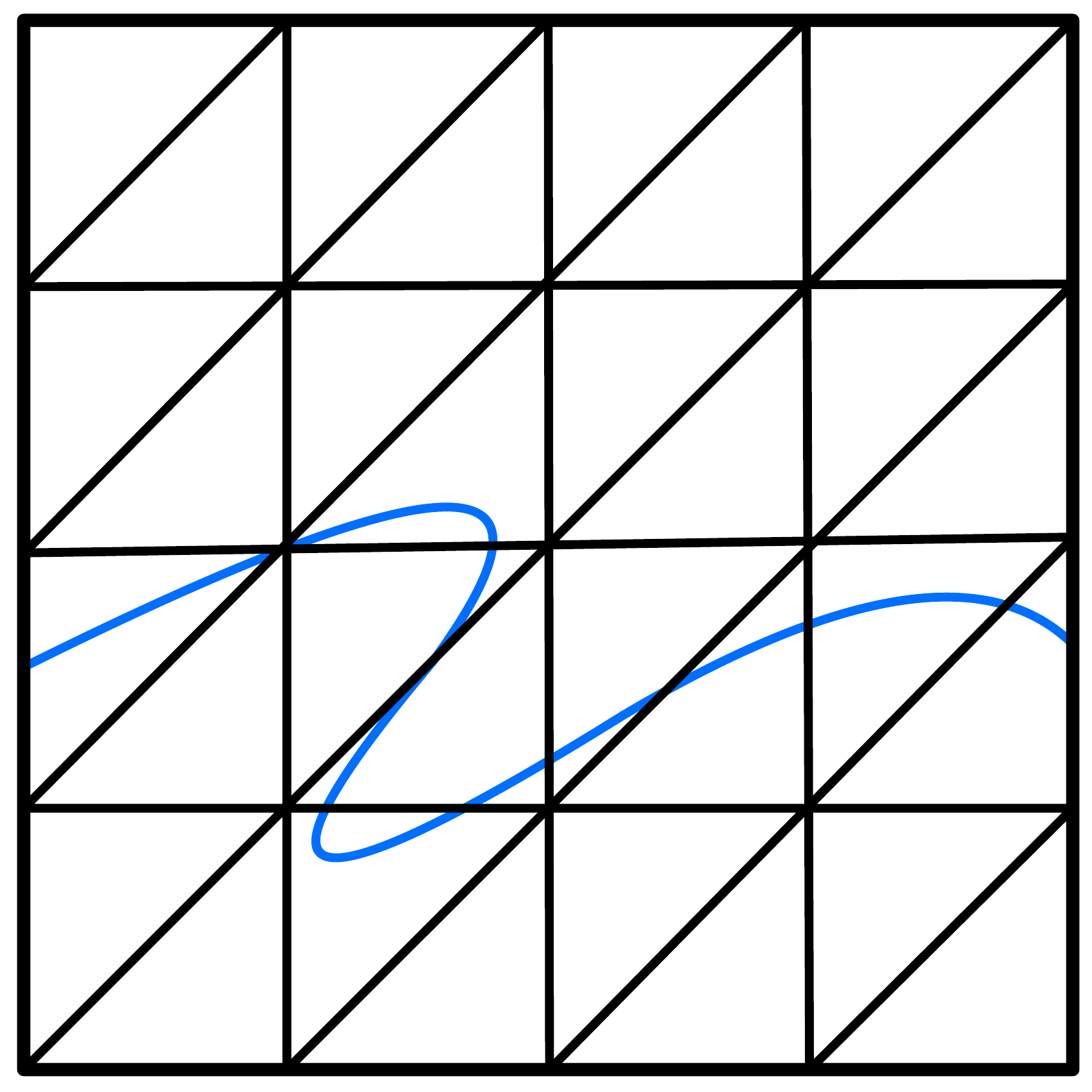}
		\hspace{5mm}
		\includegraphics[width=0.25\linewidth]{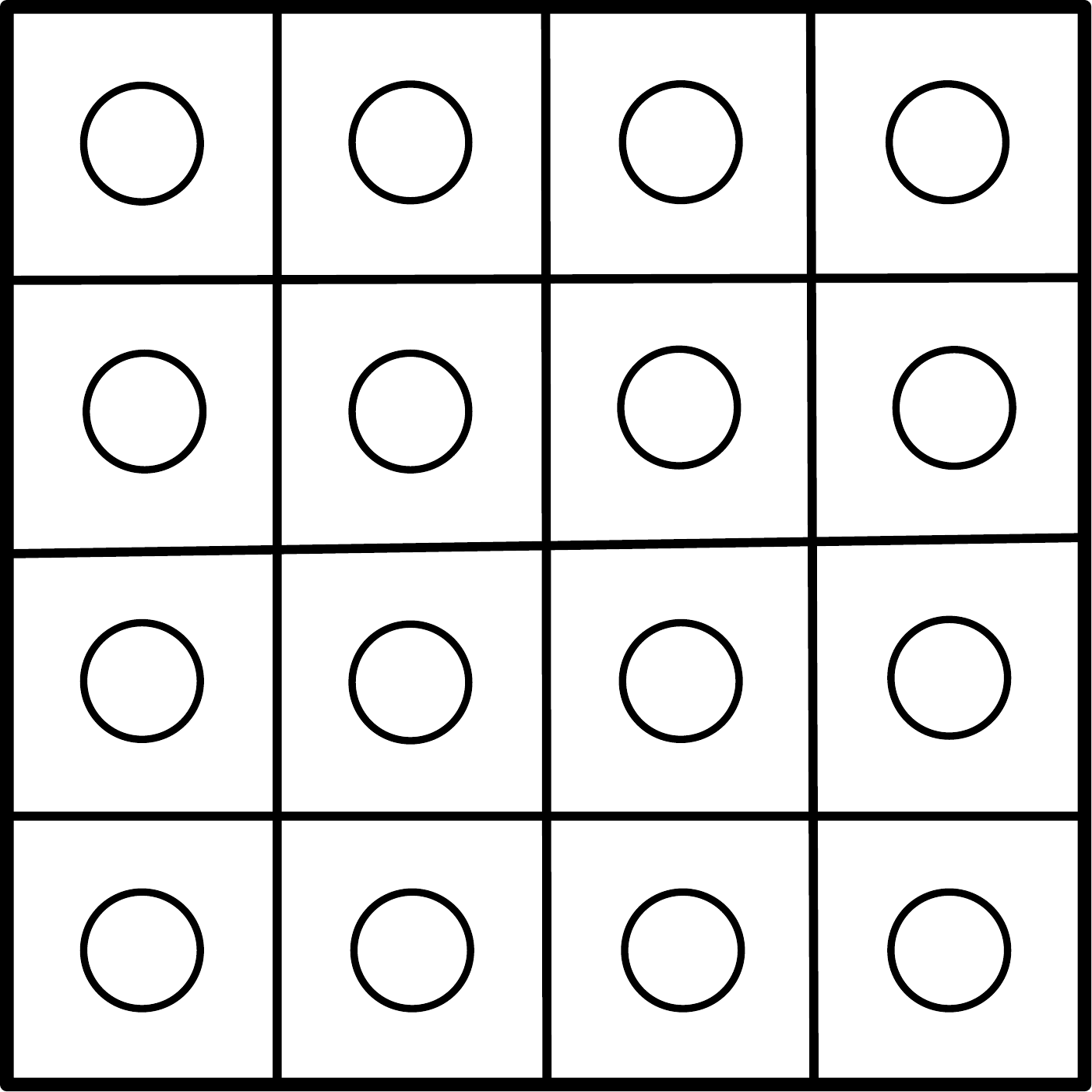}
		\hspace{5mm}
		\includegraphics[width=0.2\linewidth]{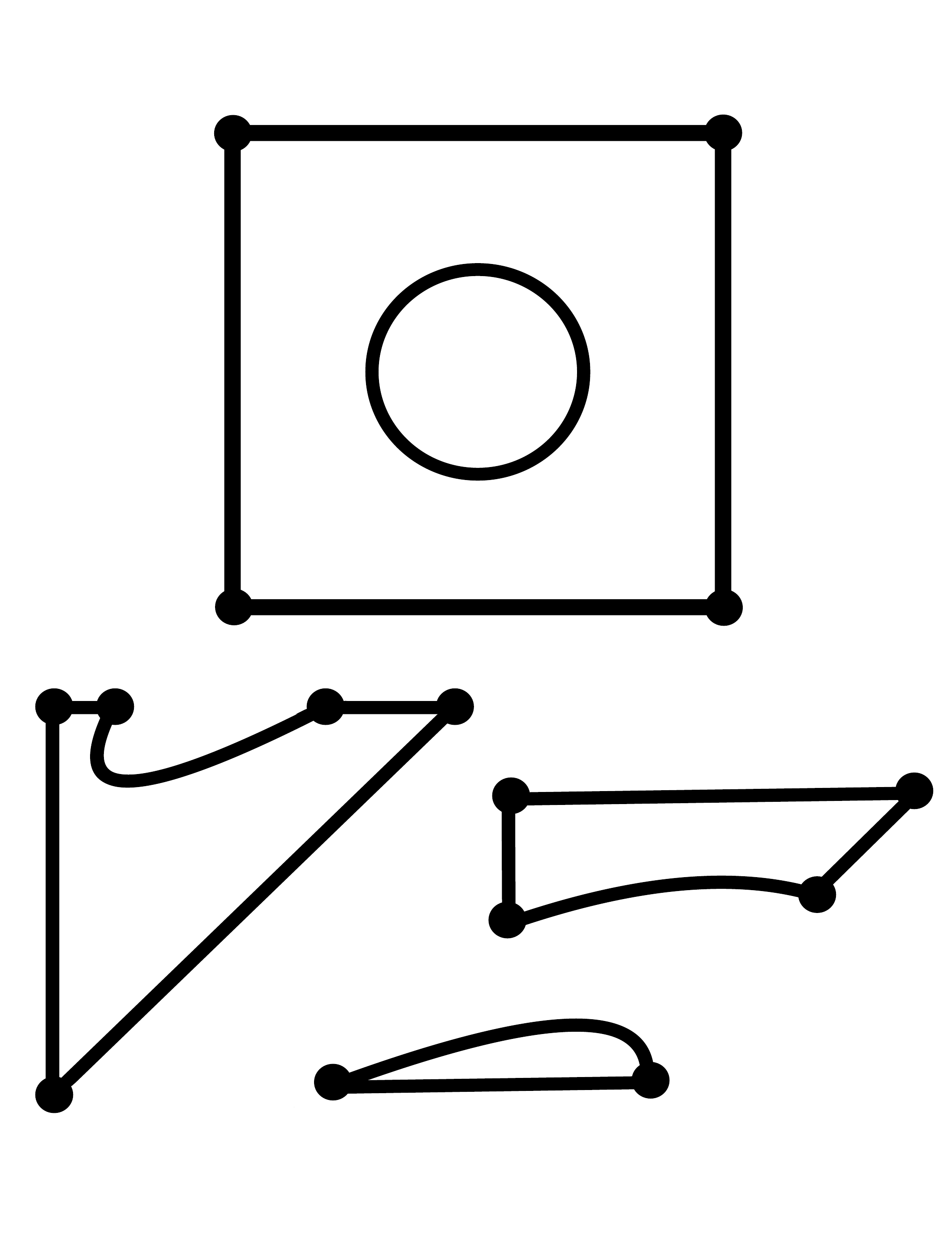}
		\caption{
			Left:
				A curvilinear mesh of a square domain featuring
				a curvilinear interface.
			Center:
				A curvilinear mesh of a square domain with
				circular punctures.
			Right:
				A few of the cells found in these meshes.
		}
		\label{mesh-examples-figure}
	\end{figure}

	Such spaces of implicitly-defined functions,
	whether arising from curvilinear, polygonal, or more
	conventional tetrahedral meshes, have appeared in the literature
	frequently in the last several years.
	Many readers are likely to be familiar with Virtual Element Methods
	(VEM), which have gained signifcant popularity
	in the last decade and have a large
	body of recent publications
	\cite{
		BeiraodaVeiga2013,
		BeiraodaVeiga2014,
		BeiraodaVeiga2020,
		Brezzi2014
	},
	some of which concern employing curvilinear mesh cells \cite{
		BeiraodaVeiga2019,
		Aldakheel2020,
		Dassi2021,
		Wriggers2020
	}.
	Our approach is more closely aligned with
	Boundary Element Based Finite Element Methods (BEM-FEM) and Trefftz methods
	\cite{
		Copeland2009,
		Hakula2022,
		Hofreither2010,
		Hofreither2011,
		Hofreither2016,
		Weisser2011,
		Weisser2014,
		Weisser2017,
		Weisser2018,
		Weisser2019,
		Weisser2019a
	}.
	In contrast to VEM, we actually construct a basis of $V_p(K)$
	and do not require projections and so-called stabilization terms.
	Futhermore, while all computations needed for forming
	the finite element system do not involve any calculations on the
	interior of mesh cells,
	and despite these basis functions being implicitly-defined,
	we have the option of obtaining
	information about the interior values, gradient, and higher derivatives
	of basis functions on the interior of each cell,
	and therefore those of the finite element solution.
	The basic framework for our approach was proposed in \cite{Anand2020},
	where we proposed methods for construction of a basis
	that automatically preserves $H^1$ conformity,
	and proved estimates for associated interpolation operators.
	Subsequently, in \cite{Ovall2022},
	we demonstrated that practical computation of
	$H^1$ semi-inner products and $L^2$ inner products
	of functions in $V_p(K)$ are feasible
	whenever $K$ is simply connected (i.e. has no holes).
	Indeed, we showed that these volumetric integrals can be reduced to
	boundary integrals, thereby circumventing any need to develop
	2D quadratures for the unconventional geometries present in
	curvilinear meshes.
	The goal of this work is to extend these results to the case where
	$K$ is multiply connected.

	In Section \ref{simply-connected-section},
	we briefly summarize how \eqref{goal-integral-1} and \eqref{goal-integral-2}
	may be computed in the case when $K$ is simply connected.
	In Section \ref{punctured-cells-section}, we address how these calculations
	can be modified in order to accomodate multiply connected mesh cells.
	We provide a handful of numerical illustrations in
	Section \ref{numerical-experiments-section},
	and conclude in Section \ref{conclusion-section}.

\section{Simply Connected Mesh Cells}
	\label{simply-connected-section}

	Let $K$ be simply connected, and suppose that $v,w\in V_p(K)$.
	The goal of this section is to provide an overview of some of the
	techniques used to compute the integrals
	\eqref{goal-integral-1} and \eqref{goal-integral-2}.
	In particular, we will see that each of these volumetric integrals
	can be feasibly reduced to boundary integrals over $\partial K$.
	These were discussed in detail in \cite{Ovall2022},
	although since its publication we have made some improvements
	that reduce computational cost, which we present here.

	\subsection{The $H^1$ Semi-inner Product}

		Given $v,w\in V_p(K)$,
		note that $\Delta v$ and $\Delta w$ are given polynomials
		of degree at most $p-2$.
		Let $P$ and $Q$ be polynomials of degree at most $p$ satisfying
		\begin{align*}
			\Delta(v-P) = 0
			~,\quad
			\Delta(w-Q) = 0
			~.
		\end{align*}
		As pointed out in \cite{Karachik2010},
		such polynomials $P$ and $Q$ can be explicitly constructed term-by-term
		by observing that
		\begin{align}
			\label{anti-laplacian-polynomial}
			P_\alpha(x) = \frac{|x|^2}{4(|\alpha|+1)!}
			\sum_{k=0}^{\floor{|\alpha|/2}}
			\frac{(-1)^k (|\alpha|-k)!}{(k+1)!}
			\left(
				\frac{|x|^2}{4}
			\right)^k
			\Delta^k(x^\alpha)
		\end{align}
		is a polynomial anti-Laplacian of $x^\alpha$ for a multi-index $\alpha$.
		That is, $\Delta P_\alpha(x) = x^\alpha$.
		Note that, in practice, $P_\alpha$ is obtained only by manipulation
		of polynomial coefficients, and poses no computational barrier.
		The same can be said of other operations involving polynomials,
		such as gradients, etc.

		Since the functions
		\begin{align*}
			\phi = v-P~,
			\quad
			\psi = w-Q
		\end{align*}
		are harmonic, we have the expansion
		\begin{align}
			\label{h1-semi-inner-product}
			\begin{split}
				\int_K \nabla v\cdot\nabla w~dx
				&= \int_K \nabla\phi\cdot\nabla w~dx
					+ \int_K\nabla P\cdot\nabla\psi~dx
					+ \int_K\nabla P\cdot\nabla Q~dx
				\\& = \int_{\partial K} w \,
					\frac{\partial\phi}{\partial\mathbf{n}}~ds
					+ \int_{\partial K} P \,
					\frac{\partial\psi}{\partial\mathbf{n}}~ds
					+ \int_K\nabla P\cdot\nabla Q~dx~.
			\end{split}
		\end{align}
		For the first two integrals in the final expression, the normal derivatives
		$\partial\phi/\partial\mathbf{n}$ and
		$\partial\psi/\partial\mathbf{n}$
		may be computed using the Dirichlet-to-Neumann map discussed below.
		Furthermore, $\nabla P\cdot\nabla Q$ is
		clearly a polynomial of degree at most $2p-2$.
		As noted in \cite{Antonietti2018},
		a straightforward application of the Divergence Theorem shows that
		\begin{align}
			\label{integrate-polynomial}
			\int_K x^\alpha ~dx
			=
			\frac{1}{2+|\alpha|}
			\int_{\partial K} (x\cdot\mathbf{n}) \, x^\alpha ~ds
			~.
		\end{align}
		In this fashion, we reduce the volumetric integral
		$\int_K \nabla v\cdot\nabla w~dx$
		to readily computable boundary integrals.

	\subsection{A Dirichlet-to-Neumann Map}

		Consider the problem of determining
		the normal derivative of a harmonic function $\phi$ given
		its Dirichlet trace $\phi|_{\partial K}$.
		Recall that $\hat\phi$ is a harmonic conjugate of a
		harmonic function $\phi$ whenever $\phi,\hat\phi$ are
		continuously twice differentiable on $K$ and satisfy the Cauchy-Riemann
		equations:
		\begin{align*}
			\frac{\partial\phi}{\partial x_1}
			=
			\frac{\partial\hat\phi}{\partial x_2}
			~, \quad
			\frac{\partial\phi}{\partial x_2}
			=
			-\frac{\partial\hat\phi}{\partial x_1}
			~.
		\end{align*}
		Given that $\phi$ is harmonic on a simply connected domain $K$,
		the existence of a harmonic conjugate of $\phi$ is guaranteed,
		and $\hat\phi$ is unique up to an additive constant.
		If $\partial K$ is smooth, for every $x\in\partial K$ it holds that
		\begin{align}
			\label{harmonic-conjugate-integral-equation-incomplete}
			\frac12 \, \hat\phi(x) + \int_{\partial K}
			\dfrac{\partial G(x,y)}{\partial\mathbf{n}(y)} \, \hat\phi(y)~dS(y)
			=
			\int_{\partial K} G(x,y) \,
			\dfrac{\partial\hat\phi}{\partial\mathbf{n}}(y) ~dS(y)
		\end{align}
		where $G(x,y) = -(2\pi)^{-1}\ln|x-y|$ is the fundamental
		solution of the Laplacian in $\RR^2$.
		Supposing that the boundary $\partial K$ is traversed
		counterclockwise,
		we let $\mathbf{t}$ denote the unit tangent vector
		and $\mathbf{n}$ denote the outward unit normal vector,
		so that the normal and tangential derivatives of $\phi$ and $\hat\phi$
		are related by
		\begin{align}
			\label{conjugate-tangent-normal}
			\dfrac{\partial\phi}{\partial\mathbf{n}}
			= \dfrac{\partial\hat\phi}{\partial\mathbf{t}}
			~,\quad
			\dfrac{\partial\hat\phi}{\partial\mathbf{n}}
			= - \dfrac{\partial\phi}{\partial\mathbf{t}}
			~,
		\end{align}
		from which the right-hand side in
		\eqref{harmonic-conjugate-integral-equation-incomplete}
		can be computed.
		Since $\hat\phi$ is unique only up to an additive constant,
		we impose $\int_K \hat\phi ~ ds = 0$, which we add to the left-hand
		side above to obtain
		\begin{align}
			\label{harmonic-conjugate-integral-equation}
			\frac12 \, \hat\phi(x) + \int_{\partial K}
			\left(
				\dfrac{\partial G(x,y)}{\partial\mathbf{n}(y)} + 1
			\right)
			\hat\phi(y)~dS(y)
			=
			-
			\int_{\partial K} G(x,y) \,
			\dfrac{\partial\phi}{\partial\mathbf{t}}(y) ~dS(y)
			~.
		\end{align}
		In practice, we solve this integral equation
		numerically for $\hat\phi$ on $\partial K$
		using a Nystr\"om method,
		where the right-hand side is computed using the tangential derivative
		of $\phi$, which is readily accessible from its trace
		$\phi|_{\partial K}$.
		Having obtained values of the harmonic conjugate $\hat\phi$ on
		the boundary $\partial K$, we may obtain its tangential
		derivative $\partial\hat\phi/\partial\mathbf{t}$ via
		numerical differentiation, which then yields values of
		the normal derivative $\partial\phi/\partial\mathbf{n}$.
		Indeed, if $x(t)$ is a sufficiently smooth
		parameterization of $\partial K$
		and we define $G(t) = \hat\phi(x(t))$, then
		\begin{align}
			G'(t)
			= \dfrac{\partial\hat\phi}{\partial\mathbf{t}}
			(x(t)) \, |x'(t)|
			= \dfrac{\partial\phi}{\partial\mathbf{n}}
			(x(t)) \, |x'(t)|
			~.
		\end{align}
		Since $G(t)$ is periodic, a natural choice to obtain $G'(t)$
		is to write a Fourier expansion
		$G(t) = \sum_{k=-\infty}^\infty \omega_k \, e^{\ii k t}$
		and obtain an approximation of $G'(t)$ by truncating the series
		\begin{align}
			\label{fft-derivative}
			G'(t) = \sum_{k=-\infty}^\infty
			\ii \,k \,\omega_k \, e^{\ii k t}~.
		\end{align}
		In practice, a Fast Fourier Transform (FFT) may be used on
		a discretization of $G(t)$, and an inverse FFT used on the coefficients
		$\ii k\omega_k$ to obtain the approximate values of $G'(t)$.

		Details of such calculations,
		including the case where $\partial K$ has corners,
		are discussed in \cite{Ovall2018}.

		\subsection{The $L^2$ Inner Product}

			Let $v = \phi + P$ and $w = \psi + Q$ be as above.
			We have the expansion
			\begin{align}
				\label{L2-expansion}
				\int_K v\,w~dx
				&= \int_K \phi\,\psi~dx
				+ \int_K Q \,\phi~dx
				+ \int_K P\,\psi~dx
				+ \int_K P\,Q~dx
				~.
			\end{align}
			Notice that the last integral can be computed with
			\eqref{integrate-polynomial},
			whereas the two middle integrals have the form
			\begin{align*}
				\int_K r \, \eta ~dx
			\end{align*}
			where $r$ is a polynomial and $\eta$ is harmonic.
			Using \eqref{anti-laplacian-polynomial},
			let $R$ be a polynomial such that
			\begin{align*}
				\Delta R = r
				~,
			\end{align*}
			then applying Green's Second Identity, we have
			\begin{align}
				\label{L2-harmonic-with-polynomial}
				\int_K r \, \eta ~dx
				=
				\int_K \eta \, \Delta R ~dx
				&= \int_{\partial K} \left[
					\eta \, \dfrac{\partial R}{\partial\mathbf{n}}
					- R \, \dfrac{\partial\eta}{\partial\mathbf{n}}
				\right]
				ds ~.
			\end{align}
			The remaining integral to be computed in \eqref{L2-expansion}
			is the $L^2$ inner product of the two harmonic functions
			$\phi$ and $\psi$.
			Toward this end,
			suppose that $\Phi$ is an anti-Laplacian of the harmonic function
			$\phi$, that is,
			\begin{align*}
				\Delta\Phi = \phi
				~.
			\end{align*}
			Then using Green's Second Identity again yields
			\begin{align}
				\label{l2-reduction}
				\int_K \phi \, \psi ~dx
				&= \int_K \psi \, \Delta\Phi ~dx
				= \int_{\partial K} \left[
					\psi \, \dfrac{\partial\Phi}{\partial\mathbf{n}}
					- \Phi \, \dfrac{\partial\psi}{\partial\mathbf{n}}
				\right]
				ds
				~.
			\end{align}
			The problem of determining such a $\Phi$,
			in particular its trace $\Phi|_{\partial K}$
			and normal derivative $\partial\Phi/\partial\mathbf{n}$,
			is addressed as follows.

		\subsection{Anti-Laplacians of Harmonic Functions}
			Notice that if $\Phi$ is an anti-Laplacian of $\phi$,
			it holds that $\Phi$ is \emph{biharmonic},
			that is, $\Delta^2\Phi=0$.
			A well-known fact (see, for example,
			pp. 269 of \cite{Garabedian1964})
			is that every biharmonic function is of the form
			\begin{align*}
				\Phi(x) = \Re \big[ \overline{z} f(z) + g(z) \big]
				~,
			\end{align*}
			where $f,g$ are some analytic functions,
			$\Re[z]$
			denotes the real part of $z\in \CC$,
			$\overline{z}$ denotes the complex conjugate,
			and we use the natural identitification of
			the complex plane with $\RR^2$
			via $x = (x_1, x_2) \mapsto z = x_1 + \ii x_2$.
			Since any anti-Laplacian of $\phi$ will suffice for our purposes,
			we will take $g = 0$ and write
			\begin{align*}
				\Phi(x) = \frac{x_1 \, \rho(x) + x_2 \, \hat\rho(x)}{4}
				~,
			\end{align*}
			where $\rho=4(\Re f)$ is a harmonic function
                        and $\hat\rho=4(\Im f)$ is a harmonic
                        conjugate of $\rho$.  It follows from the
                        Cauchy-Riemann equations that
			\begin{align*}
				\phi = \Delta\Phi
				= \dfrac{\partial\rho}{\partial x_1}
				= \dfrac{\partial\hat\rho}{\partial x_2}
				~,
			\end{align*}
			and that the gradients of $\rho$ and $\hat\rho$ must take the form
			\begin{align*}
				\nabla\rho =
				\begin{pmatrix}
					\phi \\ - \hat\phi
				\end{pmatrix}
				~,
				\quad
				\nabla\hat\rho =
				\begin{pmatrix}
					\,\,
					\hat\phi
					\;\; 
					\\
					\phi
				\end{pmatrix}
				~,
			\end{align*}
			where $\hat\phi$ is a harmonic conjugate of $\phi$.
			\begin{remark}
				Note that the gradient of $\Phi$ is given by
				\begin{align*}
					\nabla\Phi(x)=
					\dfrac14
					\begin{pmatrix}
						\rho(x)
						\\
						\hat\rho(x)
					\end{pmatrix}
					+
					\dfrac14
					\begin{pmatrix}
						x_1 & x_2
						\\
						x_2 & -x_1
					\end{pmatrix}
					\begin{pmatrix}
						\phi(x)
						\\
						\hat\phi(x)
					\end{pmatrix}
				\end{align*}
				from which we may obtain the normal derivative
				$\partial\Phi/\partial\mathbf{n}$.
			\end{remark}

			\begin{remark}
				\label{uniqueness-rho-remark}
				For any fixed constants
				$a$ and $b$, we have that
				\begin{align*}
					\frac{x_1 (\rho(x)+a) + x_2(\hat\rho(x)+b)}{4}
				\end{align*}
				is also an anti-Laplacian of $\phi$, since $(a x_1 + b x_2)/4$
				is harmonic.
			\end{remark}

			In order to compute $\rho$ and $\hat\rho$, consider the following.
			For the sake of illustration,
			we assume that the boundary $\partial K$ is smooth,
			but a similar approach works when $\partial K$ is
			only piecewise smooth, using some minor modifications.
			Given the traces of $\phi$ and $\hat\phi$ on $\partial K$,
			we have access to the tangential derivative of $\rho$ via
			\begin{align*}
				\dfrac{\partial\rho}{\partial\mathbf{t}}
				=
				\begin{pmatrix}
					\phi \\ - \hat\phi
				\end{pmatrix}
				\cdot\mathbf{t}
				~.
			\end{align*}
			Given a sufficiently smooth
			parameterization $x(t):[0,2\pi]\to\partial K$
			of the boundary, we define $g:[0,2\pi]\to\RR$ by
			\begin{align*}
				g(t)
				= \dfrac{\partial\rho}{\partial\mathbf{t}}
				\big( x(t) \big)
				\, |x'(t)|
				~.
			\end{align*}
			By the Fundamental Theorem of Calculus,
			we may obtain an anti-derivative $G$ via
			\begin{align*}
				G(t) = \int_0^t  g(\tau)~d\tau
				= \int_0^t  \nabla\rho(x(\tau))\cdot x'(\tau)~d\tau
				= \rho(x(t))
				~.
			\end{align*}
			Note that $g$ is $2\pi$-periodic and admits a Fourier expansion
			\begin{align*}
				g(t) = \sum_{k=-\infty}^\infty \omega_k \, e^{\ii k t}
				~.
			\end{align*}
			As mentioned above,
			in practice we truncate this series and
			compute the Fourier coefficients $\omega_k$
			using an FFT.
			Integrating termwise yields
			\begin{align*}
				\rho(x(t))
				= G(t)
				= C + \omega_0 t +
				\sum_{\substack{k=-\infty\\k\neq0}}^\infty
				\frac{\omega_k}{\ii k} \, e^{\ii k t}
			\end{align*}
			for an arbitrary constant $C$.
			In light of Remark \ref{uniqueness-rho-remark},
			we may pick $C$ arbitrarily;
			for instance, choose $C=0$.
			Moreover, since $\rho(x(t))$ is $2\pi$-periodic, we also see that
			$\omega_0=0$.
			In the computational context, we apply an inverse FFT to the
			coefficients $-\ii \, \omega_k / k$ in order to obtain approximate
			values of $\rho$ on the boundary.
			We apply an analogous procedure to obtain values of $\hat\rho$
			on the boundary, using the tangential derivative data
			\begin{align*}
				\dfrac{\partial\hat\rho}{\partial\mathbf{t}}
				=
				\begin{pmatrix}
					\,\,
					\hat\phi
					\;\; 
					\\
					\phi
				\end{pmatrix}
				\cdot\mathbf{t}
			\end{align*}
			and computing an anti-derivative of
			\begin{align*}
				\hat g(t) = \dfrac{\partial\hat\rho}{\partial\mathbf{t}}
				\bigl(x(t)\bigr) \, |x'(t)|
				~.
			\end{align*}

	\subsection{Piecewise Smooth Boundaries}

		In our discussion so far, we have assumed the cell boundary
		$\partial K$ to be smooth, but here we will briefly address how
		the calculations described above can be modified when
		$\partial K$ has one or more corners.
		Elaboration on these modifications can be found in
		\cite{Ovall2018}.

		Each edge of the mesh cell $K$ is discretized into $2n+1$ points,
		including the endpoints, so that the boundary is discretized into
		$N = 2n \times (\#~\text{edges of}~K)$ points,
		with redundant endpoints being neglected.
		When an edge $e \subseteq \partial K$ is a $C^2$ smooth closed
		contour, the boundary points are assumed to be sampled according a
		\emph{strongly regular parameterization} $x(t)$ of $e$
		(i.e. $|x'(t)| \geq \delta$ for all $x(t)\in e$ and for some
		fixed $\delta >0$).
		In the case where $e$ terminates at a corner, we employ a
		\emph{Kress reparameterization}, defined as follows
		(cf. \cite{Kress1990}).
		Suppose that $x(t)$ is a strongly regular parameterization of $e$ for
		$t\in[0,2\pi]$,
		then define $\tilde x(u) = x(\tau(u))$ using
		\begin{align*}
			\tau(u) = \frac{2\pi[c(u)]^\sigma}
			{[c(u)]^\sigma + [1 - c(u)]^\sigma}~,
			\quad
			c(u) = \left(\frac12 -\frac 1\sigma\right)
			\bigg(\frac{u}{\pi} - 1\bigg)^2 +
			\frac1\sigma\bigg(\frac{u}{\pi} - 1\bigg) + \frac12
			~,\quad
			u \in [0, 2\pi]
		\end{align*}
		where the \emph{Kress parameter} $\sigma \geq 2$ is fixed.
		The Kress reparameterization is not regular, with $\tilde x'(u)$
		vanishing at the endpoints.
		Indeed, $\tau'(u)$ has roots at $0$ and $2\pi$ of order
		$\sigma - 1$, which leads to heavy sampling of the boundary
		near corners.
		This effect is amplified for larger values of $\sigma$.

		Recall that whenever $\eta$ has a sufficiently smooth Dirichlet trace,
		we can compute the \emph{weighted tangential derivative}
		\begin{align*}
			\frac{d}{dt} \eta(x(t)) = \frac{\partial \eta}{\partial \mathbf{t}}
			(x(t)) \, |x'(t)|
		\end{align*}
		by using, for instance, the FFT-based approach described by
		\eqref{fft-derivative}.
		Replacing $x(t)$ with a Kress reparameterization leads to difficulty
		in recovering the values of the tangential derivative from the
		weighted tangential derivative.
		The same can be said for the \emph{weighted normal derivative}
		\begin{align*}
			\frac{\partial\eta}{\partial\mathbf{n}}(x(t)) \, |x'(t)|
			~.
		\end{align*}
		Notice, though, that for the sake of computing the boundary integral
		\begin{align*}
			\int_{\partial K} \omega \,
			\frac{\partial \eta}{\partial\mathbf{n}} ~ds
			=
			\int_{t_0}^{t_f} \omega(x(t)) \,
			\frac{\partial \eta}{\partial \mathbf{t}}(x(t)) \,
			|x'(t)|~dt
		\end{align*}
		the weighting term $|x'(t)|$ appears in the Jacobian anyway,
		so it is natural to keep the tangential and normal derivatives
		in their weighted form, inlcuding the case when using a Kress
		reparameterization.

		Note that, for the sake of effectively applying an FFT,
		we assume that the parameter $t$ is sampled at equispaced nodes
		$t_k = h k$, $h = \pi / n$, $0\leq j \leq 2n+1$,
		and likewise for the parameter $u$ when using a Kress
		reparameterization.

	\subsection{Summary of Simply Connected Case}

		Thus far, we have all the necessary tools to compute the
		goal integrals
		\eqref{goal-integral-1} and \eqref{goal-integral-2}
		in the case where $K$ is simply connected.
		It is worth reiterating that both of these volumetric integrals
		have been successfully reduced to contour integrals
		along the boundary $\partial K$, and there is no need for
		2-dimensional quadratures as \textit{all necessary computations occur
		only on $\partial K$}.

		We have the option, though, of obtaining interior values of
		$v\in V_p(K)$ as follows. Write $v = \phi + P$ as above,
		and determine a harmonic conjugate $\hat\phi$ of $\phi$.
		Then $f = \phi + \ii \hat \phi$ is an analytic function,
		and for any fixed interior point $z=x_1+\ii x_2\in K$ we have
		Cauchy's integral formula
		\begin{align}
			\label{cauchy}
			f(z)
			= \frac{1}{2\pi\ii} \oint_{\partial K}
			\frac{f(\zeta)}{\zeta-z} ~d\zeta
			~.
		\end{align}
		Furthermore, we can obtain interior values of $\nabla\phi$
		by observing that
		\begin{align*}
			f'
			= \dfrac{\partial\phi}{\partial x_1}
			- \ii \dfrac{\partial\phi}{\partial x_2}
			~,
			\quad
			f'(z)
			= \frac{1}{2\pi\ii} \oint_{\partial K}
			\frac{f(\zeta)}{(\zeta-z)^2} ~d\zeta
			~.
		\end{align*}
		Interior values of higher derivatives, such as the components of the
		Hessian, can be obtained in similar fashion if so desired.

		We conclude this section
		with a few remarks about computational complexity.
		Assume the boundary $\partial K$ is parameterized and then discretized
		using $N$ points.
		The Nystr\"om system resulting from
		\eqref{harmonic-conjugate-integral-equation}
		is dense, though well-conditioned,
                and simple linear solvers come with a computational cost
		$\mathcal O(N^3)$.
		Using more sophisticated methods, such as GMRES,
		make an improvement, but in general will never be better than
		$\mathcal O(N^2)$.
                Although even more sophisticated methods, such as
                those based on
                hierarchical matrices~\cite{Hackbusch2015} or
                hierarchical semiseparable matrices~\cite{Xia2010},
                can reduce the computational complexity even further,
                for relatively small problems such as those considered
                here, GMRES is sufficient.

		The FFT calls used for numerical differentiation have a computational
		cost of $\mathcal O(N\,\log N)$, and integration along $\partial K$
		using (using, say, the trapezoid rule) take $\mathcal O(N)$ operations.
		Operations on polynomials, such as computing anti-Laplacians,
		can be performed by manipulation of the coefficients and do not
		meaningfully contribute to the computational cost.
		So despite the many terms we have encountered in the expansion of the
		integrals \eqref{goal-integral-1} and \eqref{goal-integral-2},
		in practice these expansions are relatively cheap in comparison to
		the cost of obtaining the trace of the harmonic conjugate.
		Note that the latter computation need only happen once for each
		function $v\in V_p(K)$ considered.

		Additionally, we can use the notion of trigonometric interpolation
		to reduce computational cost even further,
		as was explored in \cite{Ovall2018}.
		With the boundary discretized into $N$ points, we can solve the
		the Nystr\"om system obtained from
		\eqref{harmonic-conjugate-integral-equation}
		as usual to obtain the harmonic conjugate $\hat\phi$.
		While performing numerical differentiation with FFT as proposed,
		we have the Fourier coefficients at our disposal, which allows
		for rapid interpolation to, say, $M = 2^m N$ points.
		We then compute the boundary integrals obtained from expanding
		\eqref{goal-integral-1} and \eqref{goal-integral-2} using
		standard 1D quadratures
		(e.g. the trapezoid rule, Simpson's rule, etc.)
		on the larger collection of $M$ points.
		The heuristics presented in \cite{Ovall2018} suggest that
		similar levels of accuracy are achieved as would be
		in the case where all $M$ sampled points are used
		for solving the Nystr\"om system.

		In the next two sections, we address how our approach can be
		modified to accomodate multiply connected mesh cells.

\section{Punctured Cells}
	\label{punctured-cells-section}

	We now consider the case with $K$ being multiply connected.
	That is, we take $K_0,K_1,\dots,K_m\subset\RR^2$ to be simply connected,
	open, bounded regions, such that:
	\begin{enumerate}
		\item for each $1\leq j \leq m$, we have that
			$\overline{K}_j$ is a proper subset
			of $K_0$---that is,	$\overline{K}_j\subset K_0$;
		\item for each $1\leq i < j \leq m$, the closures of $K_i$
			and $K_j$ are disjoint---that is,
			$\overline{K}_i\cap\overline{K}_j=\varnothing$.
	\end{enumerate}
	Additionally, we will require that
	for each $0\leq j \leq m$, the boundary $\partial K_j$
	is piecewise $C^2$ smooth without slits or cusps.
	We then take $K$ to be the region
	\begin{align*}
		K = K_0 \setminus \bigcup_{j=1}^m \overline{K}_j
		~.
	\end{align*}
	We refer to $K_j$ as the $j$th \emph{hole} (or \emph{puncture}) of $K$.
	We sometimes call $\partial K_0$ the \emph{outer boundary} of $K$,
	and $\partial K_j$ the $j$th \emph{inner boundary}.
	The outer boundary is assumed to be oriented counterclockwise,
	and the inner boundaries oriented clockwise,
	with the unit tangential vector $\mathbf{t}$,
	wherever it is defined, oriented accordingly.
	The outward unit normal $\mathbf{n}$ is therefore always a $\pi/2$
	clockwise rotation of $\mathbf{t}$.

	In the simply connected case, we made liberal use of the notion of
	harmonic conjugates. However, in multiply connected domains,
	a given harmonic function is not guaranteed to have a harmonic
	conjugate, e.g. $\ln|x|$ on an annulus centered at the origin.
	The following theorem, which is proved in \cite{Axler1986},
	for example, provides a very helpful characterization of which
	harmonic functions have a harmonic conjugate.

	\begin{theorem}
		[Logarithmic Conjugation Theorem]
		\label{LCT}
		For each of the $m$ holes of a multiply connected domain $K$,
		fix a point $\xi_j\in K_j$. Suppose that $\phi$ is a harmonic
		function on $K$. Then there are real constants $a_1,\dots,a_m$
		such that, for each $x\in K$,
		\begin{align*}
			\phi(x) = \psi(x) + \sum_{j=1}^m a_j \ln|x-\xi_j|
		\end{align*}
		where $\psi$ is the real part of an
		analytic function. In particular, $\psi$ has a harmonic conjugate
		$\hat\psi$.
	\end{theorem}

	To simplify the notation in what is to come, it will be convenient to
	define
	\begin{align}
		\label{log-lambda}
		\lambda_j(x) = \ln|x-\xi_j|~,
		\quad x\in K~,
		\quad 1\leq j\leq m~.
	\end{align}
	In a minor notational shift from Section \ref{simply-connected-section},
	note that, in this section, we will reserve $\psi$ to represent a
	``conjugable part'' of a harmonic function $\phi$, rather than treat
	$\phi$ and $\psi$ as independent harmonic functions as we did in the
	previous section.

	\subsection{A Dirichlet-to-Neumann Map for Punctured Cells}

		Our present goal is to determine the coefficients
		$a_1,\dots,a_m$, as in the statement of Theorem \ref{LCT},
		given the trace of a harmonic function $\phi$.
		We will see that we simultaneously determine $\hat\psi$
		by solving an integral equation similar to
		\eqref{harmonic-conjugate-integral-equation},
		and thereby arrive at the Dirichlet-to-Neumann map
		\begin{align}
			\label{normal-derivative-multiply connected}
			\phi|_{\partial K} \mapsto
			\dfrac{\partial\phi}{\partial\mathbf{n}}
			=
			\dfrac{\partial\hat\psi}{\partial\mathbf{t}}
			+ \sum_{j=1}^m a_j \frac{\partial \lambda_j}{\partial\mathbf{n}}
			~.
		\end{align}

		Assume for now that the boundary $\partial K$ is $C^2$ smooth.
		The case with corners is handled with Kress reparameterization,
		as discussed in the previous section.
		Our current task is to generalize the technique described by
		\eqref{harmonic-conjugate-integral-equation}
		in the case where $K$ is multiply connected.
		An alternative approach to the method we discuss here
		is presented in \cite{Greenbaum1993},
		which is comparable to our method in terms of cost and accuracy when all
                boundary edges are smooth, but does not achieve
                similar levels of accuracy when corners are present.

		Let $\hat\psi$ denote a harmonic conjugate of $\psi$
		satisfying $\int_{\partial K} \hat\psi ~ ds = 0$.
		Just as in the simply connected case, we have
		\begin{align*}
			\frac12 \, \hat\psi(x) + \int_{\partial K}
			\left(
				\dfrac{\partial G(x,y)}{\partial\mathbf{n}(y)} + 1
			\right)
			\hat\psi(y)~dS(y)
			=
			\int_{\partial K} G(x,y) \,
			\dfrac{\partial\hat\psi}{\partial\mathbf{n}}(y) ~dS(y)
			~.
		\end{align*}
		Making the replacement
		\begin{align*}
			\dfrac{\partial\hat\psi}{\partial\mathbf{n}}
			=
			- \dfrac{\partial\psi}{\partial\mathbf{t}}
			=
			- \dfrac{\partial\phi}{\partial\mathbf{t}}
			+ \sum_{j=1}^m a_j \dfrac{\partial\lambda_j}{\partial\mathbf{t}}
		\end{align*}
		and rearranging yields
		\begin{align}
			\label{harmonic-conjugate-integral-equation-multiply connected-1}
			\begin{split}
				\frac12 \, \hat\psi(x) + \int_{\partial K}
				\left(
					\dfrac{\partial G(x,y)}{\partial\mathbf{n}(y)} + 1
				\right)
				\hat\psi(y)~dS(y)
				- \sum_{j=1}^m a_j \int_{\partial K} G(x,y) \,
				\dfrac{\partial\lambda_j}{\partial\mathbf{t}}(y) ~dS(y)
				\\[6pt]
				=
				- \int_{\partial K} G(x,y) \,
				\dfrac{\partial\phi}{\partial\mathbf{t}}(y) ~dS(y)
				~.
			\end{split}
		\end{align}
		This integral equation is underdetermined due to the $m$ additional
		degrees of freedom $a_1,\dots,a_m$ in contrast to
		\eqref{harmonic-conjugate-integral-equation}.
		To resolve this, we multiply both sides of
		$\phi = \psi + \sum_{j=1}^m a_j \lambda_j$ by the normal derivative
		$\partial\lambda_\ell/\partial\mathbf{n}$ and integrate over
		$\partial K$ to obtain
		\begin{align*}
			\int_{\partial K} \phi \,
			\dfrac{\partial\lambda_\ell}{\partial\mathbf{n}}~ds
			=
			\int_{\partial K} \psi \,
			\dfrac{\partial\lambda_\ell}{\partial\mathbf{n}}~ds
			+
			\sum_{j=1}^m a_j
			\int_{\partial K} \lambda_j \,
			\dfrac{\partial\lambda_\ell}{\partial\mathbf{n}}~ds
			~.
		\end{align*}
		Invoking Green's Second Identity and the Cauchy-Riemann equations yields
		\begin{align*}
			\int_{\partial K} \psi \,
			\dfrac{\partial\lambda_\ell}{\partial\mathbf{n}}~ds
			=
			\int_{\partial K} \lambda_\ell \,
			\dfrac{\partial\psi}{\partial\mathbf{n}}~ds
			=
			\int_{\partial K} \lambda_\ell \,
			\dfrac{\partial\hat\psi}{\partial\mathbf{t}}~ds
			~.
		\end{align*}
		To write this in a form more conducive to computation,
		observe that the Fundamental Theorem of Calculus for contour integrals
		implies that
		\begin{align*}
			\int_{\partial K}
			\dfrac{\partial(\hat\psi\,\lambda_\ell)}{\partial\mathbf{t}}~ds
			= 0
		\end{align*}
		since $\partial K$ consists of $m+1$ closed contours.
		From the Product Rule, we obtain
		\begin{align*}
			\int_{\partial K} \lambda_\ell \,
			\dfrac{\partial\hat\psi}{\partial\mathbf{t}}~ds
			=
			- \int_{\partial K} \hat\psi \,
			\dfrac{\partial\lambda_\ell}{\partial\mathbf{t}}~ds
			~.
		\end{align*}
		Therefore, $\hat\psi$ ought to satisfy
		\begin{align}
			\label{harmonic-conjugate-integral-equation-multiply connected-2}
			- \int_{\partial K} \hat\psi \,
			\dfrac{\partial\lambda_\ell}{\partial\mathbf{t}}~ds
			+ \sum_{j=1}^m a_j
			\int_{\partial K} \lambda_j \,
			\dfrac{\partial\lambda_\ell}{\partial\mathbf{n}}~ds
			=
			\int_{\partial K} \phi \,
			\dfrac{\partial\lambda_\ell}{\partial\mathbf{n}}~ds
			~,
			\quad 1 \leq \ell \leq m
			~.
		\end{align}
		In summary, we have obtained a system of equations
		\eqref{harmonic-conjugate-integral-equation-multiply connected-1} and
		\eqref{harmonic-conjugate-integral-equation-multiply connected-2}
		for determining the trace of $\hat\psi$ on $\partial K$.
		Discretizing the boundary into $N$ points, as we did for the
		simply connected case, this system of equations yields a
		square augmented Nystr\"om system in $N + m$ variables,
		which we may solve with the same techniques used for solving
		\eqref{harmonic-conjugate-integral-equation}.
		The case when $\partial K$ has corners can be handled using
		a Kress reparameterization, just as in the simply connected case.

	\subsection{Anti-Laplacians of Harmonic Functions on Punctured Cells}
		\label{anti-laplacian-section}

		Next, we wish to construct an anti-Laplacian of a harmonic function
		\begin{align*}
			\phi = \psi + \sum_{j=1}^m a_j \lambda_j
		\end{align*}
		as in the statement of Theorem \ref{LCT}.
		It is simple to verify that
		\begin{align*}
			\Lambda_j(x) = \frac14 |x - \xi_j|^2\big(\ln|x - \xi_j|-1\big)
		\end{align*}
		is an anti-Laplacian of $\lambda_j(x) = \ln|x - \xi_j|$,
		so if $\Psi$ is an anti-Laplacian of $\psi$, then we have that
		\begin{align}
			\label{anti-laplacian-general-form}
			\Phi = \Psi + \sum_{j=1}^m a_j \Lambda_j
		\end{align}
		is an anti-Laplacian of $\phi$.
		The normal derivative can be computed using
		\begin{align*}
			\nabla\Phi = \nabla\Psi + \sum_{j = 1}^m a_j \nabla\Lambda_j
			~,
			\quad
			\nabla \Lambda_j(x) =
			\frac14\bigl(2 \ln|x - \xi_j| - 1\bigr) \, (x - \xi_j)~.
		\end{align*}
		Analogous to the simply connected case,
		we might seek potentials $\rho, \hat\rho$ of the vector fields
		\begin{align*}
			\mathbf{F} =
			\begin{pmatrix}
				\psi\\-\hat\psi
			\end{pmatrix}
			~,
			\quad
			\hat{\mathbf{F}} =
			\begin{pmatrix}
				\,\,
				\hat\psi
				\;\;
				\\
				\psi
			\end{pmatrix}
			~.
		\end{align*}
		While $\mathbf{F}$ and $\hat{\mathbf{F}}$ both have vanishing curls,
		this is not sufficient to guarantee that they are both conservative
		on a multiply connnected domain---a simple counterexample being
		\begin{align*}
			\psi(x) = \frac{x_1}{|x|^2}~,
			\quad
			\hat\psi(x) = - \frac{x_2}{|x|^2}~,
		\end{align*}
		taken on a circular annulus centered at the origin.

		An elementary observation from complex analysis is that
		an analytic function $g = \psi + \ii \hat\psi$ has an
		antiderivative if and only if
		$\mathbf{F} = (\psi, -\hat\psi)$ and
		$\hat{\mathbf{F}} = (\hat\psi, \psi)$ are conservative
		vector fields.
		Indeed, $G = \rho + \ii\hat\rho$ satisfies $G' = g$
		for $\nabla\rho = \mathbf{F}$ and
		$\nabla\hat\rho = \hat{\mathbf{F}}$.
		This simple fact inspires us to decompose $g$ as
		$g_0 + g_1$, where $g_0$ has an antiderivative
		and the real part of $g_1$ has an anti-Laplacian which can be
		computed a priori.
		The following proposition reveals such a decomposition.
		The proof is inspired by that given in \cite{Axler1986}
		and is unlikely to surprise a reader familiar with
		elementary complex analysis,
		but we include it for the sake of completeness.

		\begin{proposition}
			Let $g = \psi + \ii \hat\psi$ be analytic on a multiply
			connected domain $K$.
			Let $\zeta_j \in K_j$
			denote a point fixed in the $j$th hole of $K$.
			Then there are complex constants $\alpha_j\in \CC$ such that
			\begin{align*}
				g(z) = g_0(z) + \sum_{j = 1}^m
				\frac{\alpha_j}{z - \zeta_j}~,
			\end{align*}
			where $g_0$ has an antiderivative.
		\end{proposition}

		\begin{proof}
			\label{decompose-psi-prop}
			For each $1\leq j\leq m$, let
			\begin{align*}
				\alpha_j = -\frac{1}{2\pi\ii}
				\varointclockwise_{\partial K_j} g~dz
			\end{align*}
			where $\partial K_j$ is the boundary of the $j$th hole
			traversed clockwise, and define
			\begin{align*}
				g_0(z) := g(z) - \sum_{j = 1}^m
				\frac{\alpha_j}{z - \zeta_j}
				~.
			\end{align*}
			We wish to show that $g_0$ has an antiderivative,
			i.e. there is an analytic function $G_0$ for which $G_0' = g_0$.
			It will suffice to show that
			$\oint_\gamma g_0 ~dz = 0$
			for any closed contour $\gamma$ in $K$.
			To see why, fix any point $z_0 \in K$ and define
			\begin{align*}
				G_0(z) = \int_{\gamma(z_0, z)} g_0(\zeta)~d\zeta
			\end{align*}
			where $\gamma(z_0, z)$ is any contour starting at $z_0$ and
			terminating at $z \in K$.
			Since this integral would be path independent,
			it would hold that $G_0$ is well defined and $G_0' = g_0$.

			Let $\gamma$ be a closed contour in $K$.
			If $\gamma$ is homotopic to a point, then
			\begin{align*}
				\oint_\gamma g_0 ~ dz
				= \oint_\gamma g ~dz - \sum_{j = 1}^m
				\oint_\gamma \frac{\alpha_j}{z - \zeta_j} ~dz
				= 0
			\end{align*}
			holds because $g$ and $(z - \zeta_j)^{-1}$ are
			analytic in simply connected open subset of $K$.
			If $\gamma$ is homotopic to $\partial K_\ell$,
			then by the Deformation Theorem we have
			\begin{align*}
				\oint_\gamma g_0 ~ dz
				&= \oint_{\partial K_\ell} g_0 ~dz
				\\
				&= \oint_{\partial K_\ell} g ~dz - \sum_{j = 1}^m
				\oint_{\partial K_\ell} \frac{\alpha_j}{z - \zeta_j} ~dz
				\\
				&= -2\pi\ii\alpha_\ell -
				\oint_{\partial K_\ell} \frac{\alpha_\ell}{z - \zeta_\ell} ~dz
				= 0
				~.
			\end{align*}
			Note that the same conclusion holds when $\gamma$
			is oriented opposite to $\partial K_j$.
			Finally, if $\gamma$ is any closed contour in $K$
			that is not homotopic to a point,
			it holds that $\gamma$ can be decomposed as
			a closed chain (cf. Theorem 2.4 in \cite{Lang1999})
			\begin{align*}
				\gamma \sim m_1\gamma_{\ell_1} + \cdots + m_n\gamma_{\ell_n}~,
				\quad n\leq m
			\end{align*}
			with
			$\gamma_{\ell_j}$ being homotopic to $\partial K_j$
			and $m_j$ being the winding number of $\gamma$
			with respect to $\zeta_j$.
			Then
			\begin{align*}
				\oint_\gamma g_0 ~ dz
				= \sum_{j = 1}^n m_j \oint_{\gamma_{\ell_j}} g_0 ~dz
				= 0
			\end{align*}
			holds by the previous conclusion.
			Therefore $\int_\gamma g_0~dz=0$ for any closed contour
			in $\gamma$.
			As remarked above,
			this is sufficient to show that $g_0$ has an antiderivative.
		\end{proof}

		\begin{remark}
			\label{compute-bc-remark}
			For $\alpha_j = b_j + \ii c_j$ in the proof above, we have
			\begin{align}
				\label{rational-coefficients}
				b_j = -\frac{1}{2\pi} \int_{\partial K_j}
				(\hat\psi, \psi) \cdot \mathbf{t}~ds
				~,
				\quad
				c_j = \frac{1}{2\pi} \int_{\partial K_j}
				(\psi, -\hat\psi) \cdot \mathbf{t}~ds
			\end{align}
			are real-valued contour integrals around the boundary of
			the $j$th hole.
			Again, note that the inner boundary $\partial K_j$
			is taken to be oriented clockwise,
			with $\mathbf{t}$ behaving accordingly.
		\end{remark}

		Using the above results, we may decompose $g = \psi + \ii \hat\psi$
		into one part that has an antiderivative
		\begin{align*}
			g_0 = \psi_0 + \ii \hat\psi_0
			~,
		\end{align*}
		and another part that is a linear combination of rational functions
		\begin{align*}
			\sum_{j = 1}^m \frac{\alpha_j}{z - \zeta_j}
			& =
			\sum_{j = 1}^m (b_j \mu_j - c_j \hat\mu_j)
			+ \ii
			\sum_{j = 1}^m (c_j \mu_j + b_j \hat\mu_j)
			~,
			\quad
			\begin{pmatrix}
				\mu_j \\ - \hat\mu_j
			\end{pmatrix}
			=
			\nabla\ln|x- \xi_j|
			~,
		\end{align*}
		where we use $\zeta_j = \xi_j \cdot (1,\ii)$,
		with $\xi_j$ being chosen when Theorem \ref{LCT} is applied.
		Easy calculations verify that
		\begin{align*}
			M_j(x; b_j, c_j) &=
			\frac12 (b_j, c_j) \cdot (x - \xi_j)\ln|x - \xi_j|
		\end{align*}
		satisfies $\Delta M_j = b_j \mu_j - c_j \hat\mu_j$,
		and whose normal derivative can be directly obtained from
		\begin{align*}
			\nabla M_j(x; b_j, c_j)
			=
			\frac12 (b_j \,\mu_j - c_j \, \hat\mu_j) \; (x - \xi_j)
			+\frac12 \ln|x - \xi_j| \; (b_j, c_j)
			~.
		\end{align*}

		Suppose that we have computed $b_j,c_j,~1\leq j \leq m$,
		using Remark \ref{compute-bc-remark}
		and thereby obtained
		$g_0 = \psi_0 + \ii\hat\psi_0$ as in Proposition
		\ref{decompose-psi-prop}.
		Since $g_0$ has an antiderivative, we have that the vector fields
		\begin{align*}
			\mathbf{F}_0 =
			\begin{pmatrix}
				\psi_0 \\ -\hat\psi_0
			\end{pmatrix}
			~, \quad
			\hat{\mathbf{F}}_0 =
			\begin{pmatrix}
				\hat\psi_0 \\ \psi_0
			\end{pmatrix}
		\end{align*}
		are conservative.
		Let $\rho_0,\hat\rho_0$ be their corresponding potentials,
		then it follows from the Cauchy-Riemann equations that
		$\hat\rho_0$ is a harmonic conjugate of $\rho_0$,
		and their Neumann data is supplied with
		\begin{align*}
			\dfrac{\partial\rho_0}{\partial\mathbf{n}}
			= \mathbf{F}_0 \cdot \mathbf{n}
			~,
			\quad
			\dfrac{\partial\hat\rho_0}{\partial\mathbf{n}}
			= \hat{\mathbf{F}}_0 \cdot \mathbf{n}
			~.
		\end{align*}
		The solution to the Neumann problem $\Delta\rho_0=0$,
		$\nabla\rho_0\cdot\mathbf{n}=\mathbf{F}_0\cdot\mathbf{n}$
		is unique up to an additive constant, and similarly for
		$\hat\rho_0$.
		Given the conclusion of Remark \ref{uniqueness-rho-remark},
		we may fix these constants arbitrarily,
		and we make the choice to impose
		\begin{align*}
			\int_{\partial K} \rho_0 ~ds = 0
			~,
			\quad
			\int_{\partial K} \hat\rho_0 ~ds = 0
			~.
		\end{align*}
		Using the same techniques used to solve
		\eqref{harmonic-conjugate-integral-equation},
		we may determine the traces of $\rho, \hat\rho$
		by solving
		\begin{align}
			\label{rho-integral-eqn}
			\frac12 \, \rho_0(x) + \int_{\partial K}
			\left(
				\dfrac{\partial G(x,y)}{\partial\mathbf{n}(y)} + 1
			\right)
			\rho_0(y)~dS(y)
			&=
			\int_{\partial K} G(x,y) \,
			\mathbf{F}_0(y) \cdot \mathbf{n}(y) ~dS(y)
			~,
			\\[6pt]
			\label{rho-hat-integral-eqn}
			\frac12 \, \hat\rho_0(x) + \int_{\partial K}
			\left(
				\dfrac{\partial G(x,y)}{\partial\mathbf{n}(y)} + 1
			\right)
			\hat\rho_0(y)~dS(y)
			&=
			\int_{\partial K} G(x,y) \,
			\hat{\mathbf{F}}_0(y) \cdot \mathbf{n}(y) ~dS(y)
			~.
		\end{align}
		In summary, we have the following recipe to determine an
		anti-Laplacian $\Phi$ of a given harmonic function $\phi$ on a
		multiply connected domain $K$:
		\begin{enumerate}
			\item Write $\phi = \psi + \sum_j a_j \lambda_j$ as in Theorem
			\ref{LCT}, and determine $\hat\psi$ and $a_1,\dots,a_m$ using
			\eqref{harmonic-conjugate-integral-equation-multiply connected-1}
			and
			\eqref{harmonic-conjugate-integral-equation-multiply connected-2}.
			\item Determine $b_1,\dots,b_m,c_1,\dots,c_m$ by computing
			\eqref{rational-coefficients}.
			\item Set
			\begin{align*}
				\psi_0 = \psi - \sum_{j = 1}^m (b_j \mu_j - c_j \hat\mu_j)
				~, \quad
				\hat\psi_0 = \hat\psi -
				\sum_{j = 1}^m (c_j \mu_j + b_j \hat\mu_j)
				~.
			\end{align*}
			\item Solve \eqref{rho-integral-eqn} and
			\eqref{rho-hat-integral-eqn} using
			$\mathbf{F}_0 = (\psi_0, -\hat\psi_0)$ and
			$\hat{\mathbf{F}}_0 = (\hat\psi_0, \psi_0)$.
			\item Set
			\begin{align*}
				\Phi(x) &=
				\frac14\big(x_1\rho_0(x) + x_2\hat\rho_0(x)\big)
				+ \sum_{j = 1}^m M_j(x; b_j, c_j)
				+ \sum_{j = 1}^m a_j\Lambda_j(x)
				~.
			\end{align*}
		\end{enumerate}

	\subsection{Summary of Multiply Connected Case}

		The strategies outlined in Section \ref{simply-connected-section}
		for expanding the $H^1$ semi-inner product and
		$L^2$ inner product and reducing each term to integrals along the
		boundary $\partial K$ still hold in the multiply connected case.
		The two primary ways in which the computations have changed
		in the multiply connected case is
		(i) the Dirichlet-to-Neumann map for harmonic functions,
		and
		(ii) obtaining an anti-Laplacian of a harmonic function.

		We also note that the FFT-based method \eqref{fft-derivative}
		can still be used to obtain $\partial\hat\psi / \partial\mathbf{t}$
		from $\hat\psi|_{\partial K}$,
		\emph{but must be used on each component
		$\partial K_0,\partial K_1, \dots, \partial K_m$
		of the boundary separately}.
		An astute reader may also notice that similar concerns
		would thwart our attempts to obtain $\rho_0,\hat\rho_0$ with
		an FFT-based approach when $K$ is multiply connected,
		in contrast to the simply connected case.

		Should we wish to obtain interior values of
		\begin{align*}
			v(x) = \psi(x) + P(x) + \sum_{j = 1}^m a_j \lambda_j(x)~,
		\end{align*}
		which we emphasize is optional,
		we may proceed as follows.
		The values of $P$ and $\lambda_j$ may be obtained through
		direct computation.
		To see how to obtain interior values of $\psi$,
		consider the complex contour integral of
		$f = \psi + \ii \hat\psi$ along the outer boundary $\partial K_0$.
		Let $\gamma$ be the positively oriented boundary of a
		closed disk that lies in $K$ and is centered at a fixed
		$z \in K$.
		Then $\partial K_0$ can be decomposed into the chain
		\begin{align*}
			\partial K_0 \sim \gamma - \partial K_1 - \cdots - \partial K_m
		\end{align*}
		with the inner boundaries oriented clockwise.
		So integrating and applying Cauchy's integral formula to
		$\gamma$ yields
		\begin{align*}
			\oint_{\partial K_0} \frac{f(\zeta)}{\zeta - z}~d\zeta
			=
			2\pi\ii f(z)
			- \sum_{j = 1}^m
			\oint_{\partial K_j} \frac{f(\zeta)}{\zeta - z}~d\zeta
			~,
		\end{align*}
		and rearranging gives the familiar formula
		\begin{align*}
			f(z) = \frac{1}{2\pi\ii}\oint_{\partial K}
			\frac{f(\zeta)}{\zeta - z}~d\zeta
			~,\quad
			z \in K
			~,
		\end{align*}
		provided that we orient the boundary components properly.
		The same argument can be applied to show that the
		components of the gradient can be obtained in the interior
		by applying the above result to $f'$,
		and similarly for higher derivatives.

\section{Numerical Examples}
	\label{numerical-experiments-section}

	\begin{figure}
		\centering
		\includegraphics[height=0.3\textwidth]{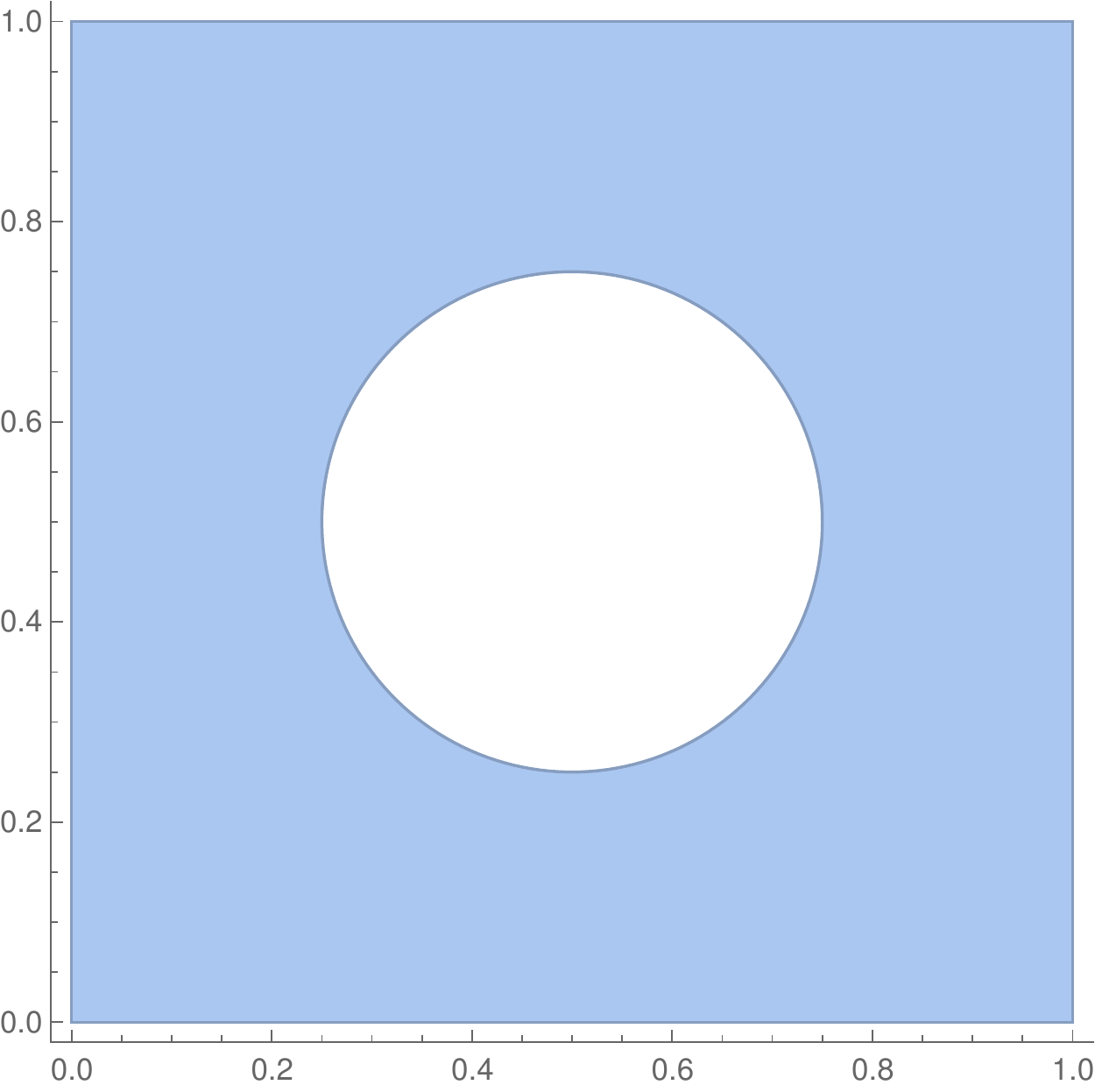}
		\hspace{5mm}
		\includegraphics[height=0.3\textwidth]{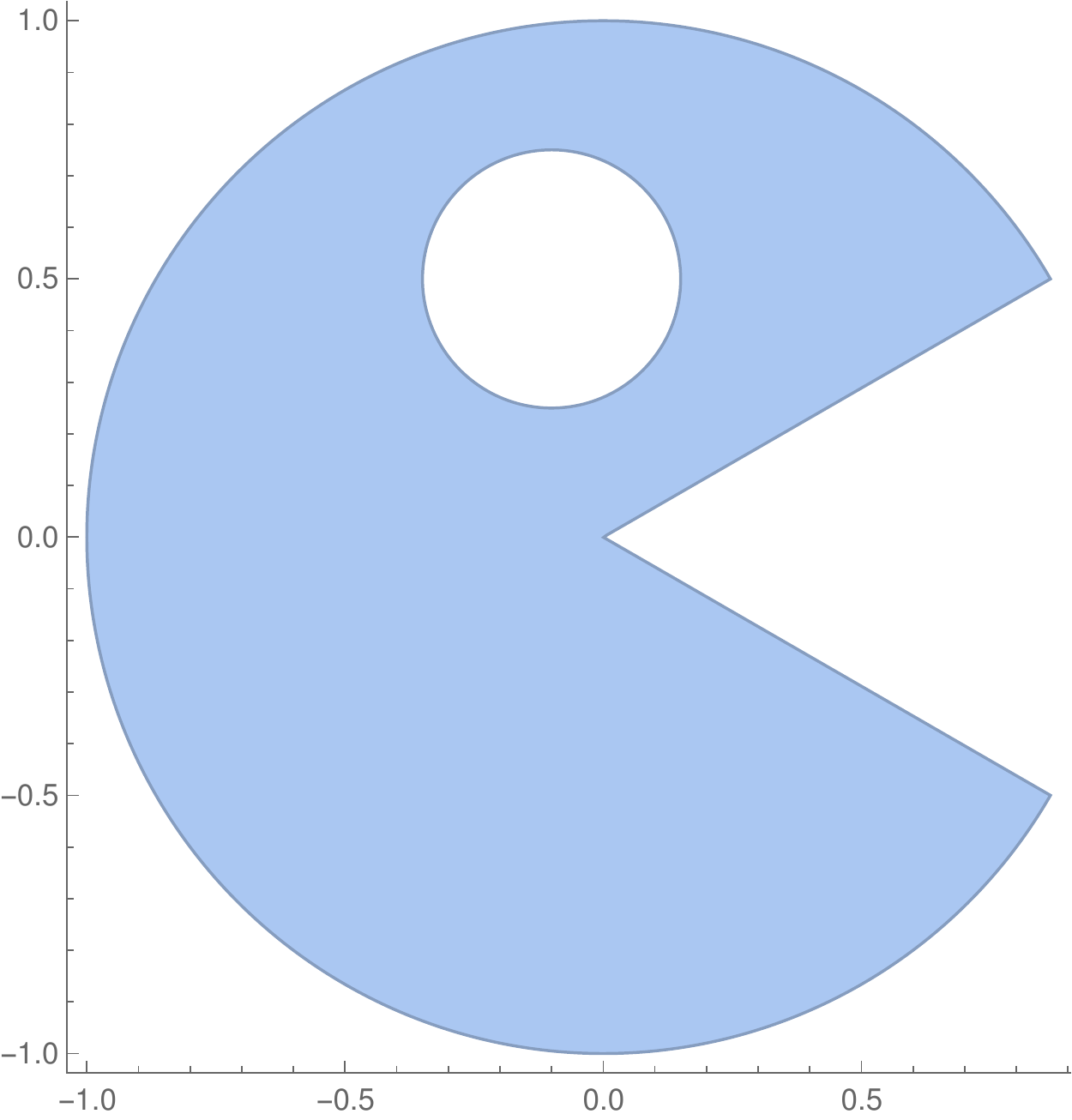}
		\hspace{5mm}
		\includegraphics[height=0.3\textwidth]{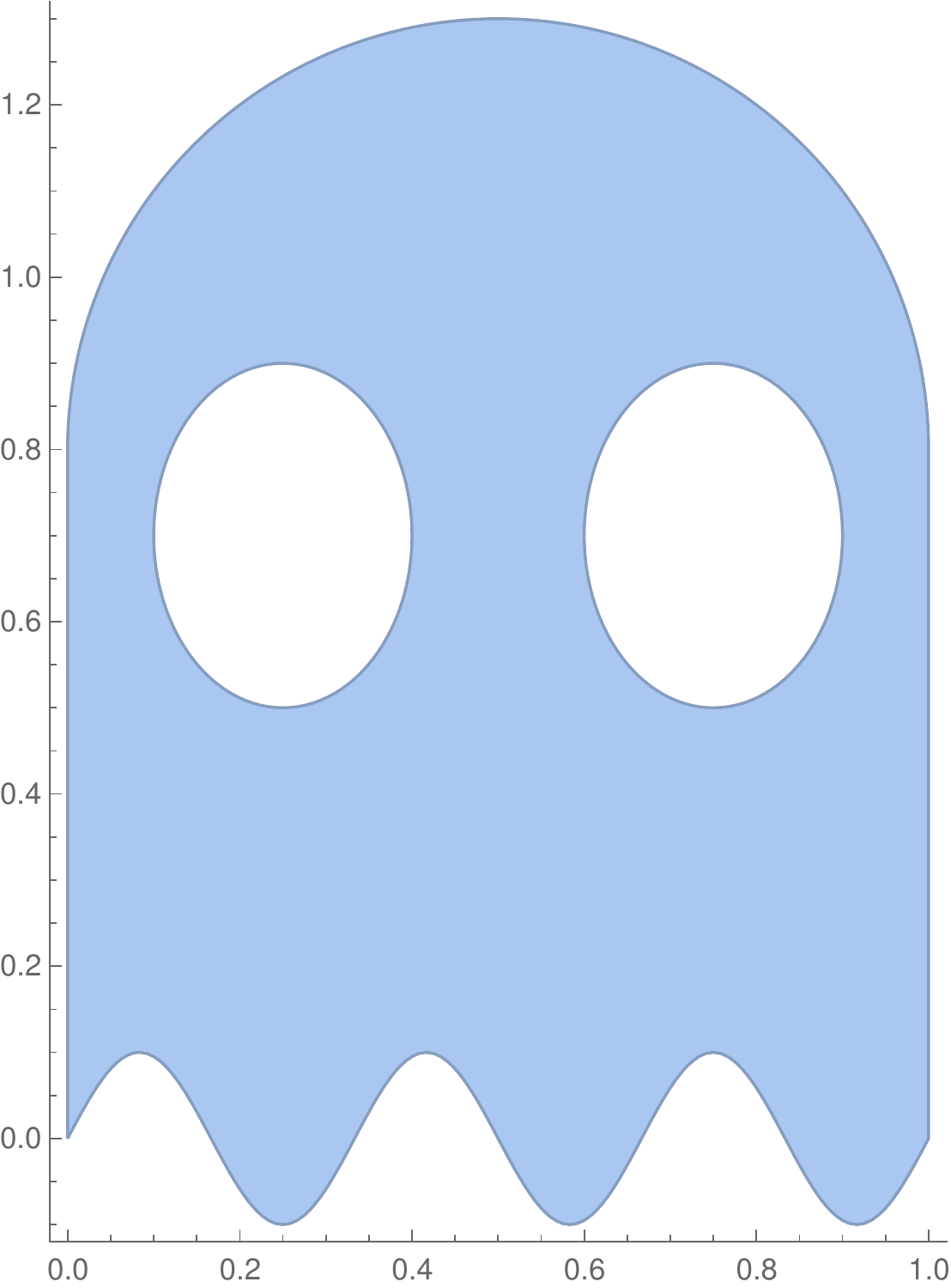}
		\caption{Punctured cells used for numerical experiments
		in Section \ref{numerical-experiments-section}.}
		\label{punctured-cells-figure}
	\end{figure}

	For each of the following examples, we pick explicit functions
	$v,w \in H^1(K)\cap C^2(K)$ of the form
	\begin{align*}
		v = \phi + P~,
		\quad
		w = \psi + Q~,
	\end{align*}
	where $\phi,\psi$ are harmonic functions and $P,Q$ are polynomials.
	While we will pick $v,w$ to be explicitly-defined, note that
	only the boundary traces $v|_{\partial K},w|_{\partial K}$
	and the coefficients of the polynomial Laplacians $\Delta v, \Delta w$
	are supplied as input for the computations.
	Using explicitly defined functions is convenient for convergence
	studies, but in practice the computations will work the same for
	implicitly-defined functions.

	Unless otherwise noted,
	we keep the Kress parameter $\sigma = 7$ fixed,
	as we observed that this value of the Kress parameter gave
	satisfactory results under a wide range of circumstances.
	Boundary integrals are evaluated by	applying the trapezoid rule.

	In each example, reference values for the $H^1$ and $L^2$
	(semi-)inner products
	were obtained with \textit{Wolfram Mathematica}.
	The mesh cell $K$ was defined using \texttt{ImplicitRegion[]},
	and volumetric integals were computed using \texttt{NIntegrate[]}.
	We remark that our implementation, albeit far from optimized,
	was significantly faster than \textit{Mathematica}
	for computing these kinds of integrals.
	(In fairness, we compute these integrals as boundary integrals,
	whereas it seems that \textit{Mathematica} implements
	general-purpose adaptive 2D quadrature over the volume,
	so perhaps the comparison in performance is unjustified.)
	For each reference value, we also give the error estimate
	that was provided by \textit{Mathematica}.

	Each of the numerical examples in this section is presented with
	Jupyter Notebook in the GitHub repository
	\begin{quote}
		\url{https://github.com/samreynoldsmath/PuncturedFEM}
	\end{quote}
	which also contains the Python source code implementing the numerical
	methods we have described in this work.

	\begin{example}
		[Punctured Square]
		\label{square-hole-example}

		\begin{table}
			\centering
			\caption{
				Errors in intermediate quantities for $v$ on
				the square with a circular hole in
				Example \ref{square-hole-example}:
				the logarithmic coefficient $a_1$,
				the trace of the harmonic conjugate $\hat\psi$,
				the weighted normal derivative wnd,
				and the trace of the anit-Laplacian $\Phi$.
				The latter three are given in the $L^2$ boundary norm.
			}
			\begin{tabular}{|c|c|c|c|c|}
				\hline
				$n$
				& $a_1$ error
				& $\hat\psi$ error
				& wnd error
				& $\Phi$ error
				\\
				\hline
4	&	1.7045e-03	&	3.5785e-02	&	2.8201e-01	&	8.3234e-03	\\
8	&	3.5531e-07	&	2.6597e-04	&	1.2855e-03	&	3.9429e-05	\\
16	&	1.0027e-09	&	1.1884e-06	&	3.7415e-06	&	3.3785e-07	\\
32	&	3.5905e-13	&	2.3095e-09	&	1.0434e-08	&	1.9430e-09	\\
64	&	1.8874e-14	&	1.6313e-12	&	6.4780e-11	&	7.0728e-12
				\\
				\hline
			\end{tabular}
			\label{square-hole-table}
		\end{table}

		Let $K_0 = (0,1) \times (0,1)$ be a unit square,
		and let $K_1 = \{x\in \RR^2 : |x - \xi| < 1/16\}$ be a disk
		of radius $1/4$ centered at $\xi = (1/2, 1/2)$.
		The cell under consideration is
		the square with the disk removed,
		$K = K_0 \setminus\overline{K}_1$,
		as depicted in the left-hand side of Figure
		\ref{punctured-cells-figure}.
		Define
		\begin{align*}
			v(x) = e^{x_1} \cos x_2 + \ln |x - \xi| +
			x_1^3 x_2 + x_1 x_2^3~.
		\end{align*}
		Notice that $v$ can be decomposed into
		\begin{align*}
			v = \phi + P~,
			\quad
			\phi(x) = e^{x_1} \cos x_2 + \ln |x - \xi|~,
			\quad
			P(x) = x_1^3 x_2 + x_1 x_2^3~,
		\end{align*}
		with $\phi$ being harmonic and the polynomial $P$ having the
		Laplacian
		\begin{align*}
			\Delta v(x) = \Delta P(x) = 12 x_1 x_2
			~.
		\end{align*}
		Furthermore, $\phi$ can be decomposed as
		\begin{align*}
			\phi(x) = \psi(x) + a_1 \ln |x - \xi_1|
		\end{align*}
		with $a_1 = 1$ and $\xi_1 = \xi = (1/2, 1/2)$.
		In Table \ref{square-hole-table},
		we report the errors in the computed approximations
		of $a_1$, $\hat\psi$, the weighted normal derivative (wnd) of
		$\phi$, and the trace of the anti-Laplacian $\Phi$.
		Since a harmonic conjugate $\hat\psi$ is unique only up to
		an additive constant, we compute the error as
		\begin{align*}
			\left(
				\int_{\partial K}
				(\hat\psi_\text{exact} - \hat\psi_\text{computed} + c)^2
				~ds
			\right)^{1/2}~,
		\end{align*}
		where $c$ is a constant minimizing the $L^2(\partial K)$ distance
		between the traces of
		$\hat\psi_\text{exact}$ and $\hat\psi_\text{computed}$,
		namely
		\begin{align*}
			c = -\frac{1}{|\partial K|}\int_{\partial K}
			(\hat\psi_\text{exact} - \hat\psi_\text{computed})
			~ds
			~.
		\end{align*}
		In general, an anti-Laplacian $\Phi$ is unique
		only up to the addition of a harmonic function, which is much
		less restrictive.
		However, we see from Remark \ref{uniqueness-rho-remark}
		that two different anti-Laplacians computed using
		the techniques described will differ by the addition of a linear
		function $c_1 x_1 + c_2 x_2$ for some constants $c_1, c_2$.
		It follows the that difference between $\Phi_\text{computed}$ and
		\begin{align*}
			\Phi_\text{exact}(x) =
			\frac14 e^{x_1}
			\big(x_1 \cos x_2 + x_2 \sin x_2\big)
			+ \frac14 |x - \xi|^2 \big(\ln|x - \xi|-1 \big)
		\end{align*}
		ought to be well-modeled by $c_1 x_1 + c_2 x_2$,
		and we choose to determine optimal constants $c_1,c_2$
		via least-squares.
		Upon doing so, we compute the error in $\Phi$ with
		\begin{align*}
			\left(
				\int_{\partial K}
				(\Phi_\text{exact} - \Phi_\text{computed}
				+ c_1 x_1 + c_2 x_2)^2
				~ds
			\right)^{1/2}~.
		\end{align*}
		In Table \ref{square-hole-table}, we list the absolute error
		in the logarithmic coefficient $a_1$,
		as well as the $L^2$ boundary norm of the errors in
		the harmonic conjugate trace $\hat\psi|_{\partial K}$,
		the weighted normal derivative of $\phi$,
		and the trace of the anti-Laplacian $\Phi$.
		We observe superlinear convergence in these quantities
		with respect to the boundary discretization parameter $n$.
		In Table \ref{h1-l2-errors-table},
		we provide the errors in the $H^1$ semi-inner product and $L^2$
		inner product of $v$ and $w$,
		where
		\begin{align*}
			w(x) = \frac{x_1 - 0.5}{(x_1 - 0.5)^2 + (y - 0.5)^2}
			+ x_1^3 + x_1 x_2^2
			~.
		\end{align*}
		The reference values
		\begin{align*}
			\int_K \nabla v \cdot \nabla w ~ dx
			&\approx 4.46481780319135
			\pm 9.9241 \times 10^{-15}
			~,
			\\
			\int_K v \, w ~dx
			&\approx 1.39484950156676
			\pm 2.7256 \times 10^{-16}
		\end{align*}
		were obtained with \textit{Mathematica}, as noted above.
		Notice that the convergence trends in these quantities parallels those
		of the intermediate quantities found in
		Table \ref{square-hole-table}.

		\begin{table}
			\centering
			\caption{
				Absolute errors in
				the $H^1$ semi-inner product and $L^2$ inner product
				for the Punctured Square
				(Example \ref{square-hole-example}),
				Pac-Man (Example \ref{pacman-example}),
				and Ghost (Example \ref{ghost-example}).
			}
			\begin{tabular}{|c||c|c||c|c||c|c||}
				\hline
				& \multicolumn{2}{c||}{Punctured Square}
				& \multicolumn{2}{c||}{Pac-Man}
				& \multicolumn{2}{c||}{Ghost}
				\\
				\hline
				$n$
				& $H^1$ error
				& $L^2$ error
				& $H^1$ error
				& $L^2$ error
				& $H^1$ error
				& $L^2$ error
				\\
				\hline
4	&	1.5180e-02	&	3.4040e-03	&	7.2078e-02	&	2.1955e-02	&	2.4336e+00	&	5.9408e-03	\\
8	&	2.6758e-04	&	8.3812e-05	&	3.3022e-02	&	5.4798e-03	&	1.0269e-02	&	1.3086e-02	\\
16	&	8.4860e-07	&	3.8993e-08	&	1.2495e-03	&	1.0159e-04	&	1.5273e-03	&	1.3783e-04	\\
32	&	1.0860e-09	&	2.8398e-11	&	6.5683e-06	&	4.6050e-07	&	5.3219e-07	&	8.1747e-07	\\
64	&	9.5390e-13	&	1.1036e-13	&	4.6834e-08	&	2.1726e-09	&	1.5430e-11	&	4.6189e-11	\\
				\hline
			\end{tabular}
			\label{h1-l2-errors-table}
		\end{table}

	\end{example}

	\begin{example}
		[Pac-Man]
		\label{pacman-example}

		For our next example, we consider the Pac-Man domain
		$K = K_0 \setminus \overline{K}_1$, where
		$K_0$ is the sector of the unit circle centered at the origin
		for $\theta_0 < \theta < 2\pi - \theta_0$,
		$\theta_0 = \pi / 6$,
		and $K_1$ is a disk of radius $1/4$ centered at $(-1/10, 1/2)$.
		(See Figure \ref{punctured-cells-figure}, center.)
		The function
		\begin{align*}
			v(x) = r^\alpha \, \sin(\alpha \theta)
		\end{align*}
		specified in polar coordinates $(r,\theta)$ is harmonic everywhere
		except possibly the origin for any fixed $\alpha > 0$.
		For the choice $0 < \alpha < 1$, we have that the gradient of
		$v$ is unbounded near the origin; indeed,
		$|\nabla v| = \alpha r^{\alpha - 1}$.
		Noting that the boundary $\partial K$ intersects the origin,
		it follows that normal derivative of $v$ is also unbounded
		near the origin.
		To test whether our strategy is viable for such functions,
		we compute the $H^1$ seminorm and $L^2$ norm of $v$ for
		$\alpha = 1/2$.
		The results are given in Table \ref{h1-l2-errors-table},
		using the reference values
		\begin{align*}
			\int_K |\nabla v|^2~dx
			&\approx 1.20953682240855912
			\pm 2.3929 \times 10^{-18}
			~,\\
			\int_K v^2 ~dx
			&\approx 0.97793431492143971
			\pm 3.6199\times 10^{-19}
			~.
		\end{align*}
		Although convergence is still rapid in this case,
		it is less so than in the previous example,
		as may be expected when considering more challenging integrands,
		as we have here.
	\end{example}

	\begin{example}
		[Ghost]
		\label{ghost-example}

		Our final example demonstrates that our method works when
		$K$ has more than one puncture, as well as when
		the boundary has
		edges that are not line segments or circlular arcs.
		The lower edge of the Ghost is the sinusoid
		$x_2 = 0.1 \sin (6\pi x_1)$ for $0 < x_1 < 1$,
		the sides are vertical line segments, the upper boundary
		is a circular arc of radius $1/2$ centered at $(0.5, 0.8)$,
		and the inner boundaries are ellipses with
		$0.15$ and $0.2$ as the semi-minor and semi-major axes,
		respectively, with one centered at $(0.25, 0.7)$ and the other
		at $(0.75, 0.7)$.
		(See Figure \ref{punctured-cells-figure}, right.)
		The functions we choose to integrate are
		\begin{align*}
			v(x) &= \frac{x_1 - 0.25}{(x_1 - 0.25)^2 + (x_2 - 0.7)^2}
			+ x_1^3 x_2 + x_2^2
			~,
			\\[12pt]
			w(x) &= \ln\big[(x_1 - 0.75)^2 + (x_2 - 0.7)^2\big]
			+ x_1^2 x_2^2 - x_1 x_2^3
			~.
		\end{align*}
		Notice that these functions have singularities in the holes
		of $K$, one rational and the other logarithmic.
		In Table \ref{h1-l2-errors-table},
		we compare the computed
		$H^1$ and $L^2$ (semi-)inner products to the reference values
		\begin{align*}
			\int_K \nabla v \cdot \nabla w ~dx
			&\approx
			-6.311053612386
			\pm 3.6161 \times 10^{-12}
			~,\\
			\int_K v \, w ~dx
			&\approx -3.277578636852
			\pm 1.0856\times 10^{-13}
			~.
		\end{align*}
		We conjecture that for $n=4$, the error in the $H^1$ semi-inner
		product is significantly worse than in
		the other two examples because this level of boundary discretization
		is insufficient to fully capture the oscillatory behavior of the
		lower edge.

		Lastly, we demonstrate the ability to obtain interior values of
		$v$ and $\nabla v$ in the interior of $K$
		in Figure \ref{ghost-intvals-figure}.
		All computations used to generate these values used the
		boundary discretization parameter $n = 64$.
		Due to the factor(s) of $\zeta - z$ in the denominator
		of the integrand in Cauchy's integral formula, where $\zeta$ is
		a point on the boundary and $z$ is the point in the interior
		where we wish to evaluate $v$,
		notice that the error in evaluation is considerably greater
		when $z$ is near the boundary.
		For this reason, we choose to not to perform
		the evaluation if $|\zeta - z| < \varepsilon$
		is found to hold for some boundary point $\zeta \in \partial K$
		and a fixed $\epsilon > 0$.
		For this example, we chose $\varepsilon = 0.02$.

		\begin{figure}
			\centering
			\includegraphics[width=0.9\textwidth]{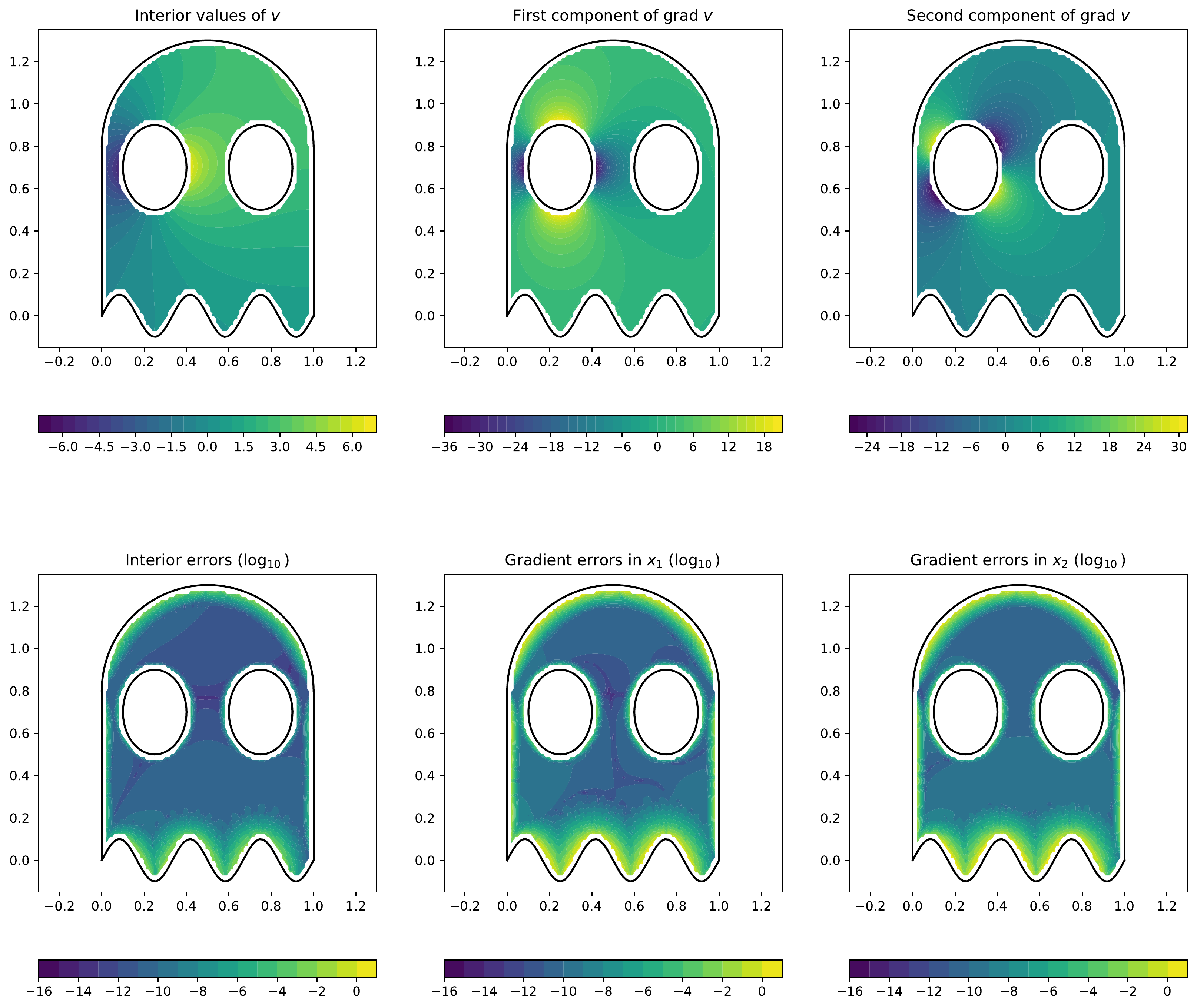}
			\caption{
				Interior values of $v$ and $\nabla v$ in
				Example \ref{ghost-example}.
				In the left column, we have the computed values of $v$
				on top, and the base 10 logarithm of the absolute error
				on bottom.
				This setup is repeated in the middle and right columns
				for the components of the gradient.
			}
			\label{ghost-intvals-figure}
		\end{figure}

	\end{example}

\section{Conclusion}
	\label{conclusion-section}


	We have seen that, given implicitly-defined functions $v,w$ of the
	type that arise in a finite element setting, we can efficiently
	compute the $H^1$ semi-inner product and $L^2$ inner product of
	$v$ and $w$ over multiply connected curvilinear mesh cells.
	All of the necessary computations occur only on the boundary of
	mesh cells, although we have the option of obtaining interior
	values of these functions and their derivatives using quantities
	obtained in the course of these calculations.
	Two key computations needed for our approach are
	(i) a Dirichlet-to-Neumann map for harmonic functions, and
	(ii) finding the trace and normal derivative of an anti-Laplacian of a
	harmonic function.
	We have described how both of these computations may be feasiblely
	accomplished on planar curvilinear mesh cells with holes.
	Numerical examples demonstrate superlinear convergence
	with respect to the number of sampled boundary points.

\section*{Funding}
	This work was partially supported by the
	National Science Foundation through
	NSF grant DMS-2012285 and
	NSF RTG grant DMS-2136228.

\section*{Code availability}

	Python code used for the experiments in this manuscript is publicly
	available at
	\url{https://github.com/samreynoldsmath/PuncturedFEM}


\bibliography{titles}
\bibliographystyle{abbrv}


\end{document}